\documentclass[preprint,12pt]{elsarticle}

\usepackage[english]{babel}
\usepackage[utf8]{inputenc}
\usepackage{amsmath}
\usepackage{amssymb}
\usepackage{amsthm}
\usepackage{amscd}
\usepackage{amsfonts}
\usepackage{newlfont}
\usepackage{enumerate}

\newtheorem{thm}{Theorem}[section]

\newtheorem{lem}[thm]{Lemma}
\newtheorem{prop}[thm]{Proposition}
\newtheorem{definition}[thm]{Definition}
\newtheorem{rem}[thm]{Remark}
\newtheorem{corol}[thm]{Corollary}

\journal{Journal of Mathematical Analysis and Applications, \,}

\begin{document}
	
	\begin{frontmatter}
		
		
		\title{On tangential transversality\tnoteref{label1}}
		\tnotetext[label1]{This work was partially supported by
			the Sofia University "St. Kliment Ohridski" fund "Research \& Development"
			under contract 80-10-133/25.04.2018  and by the Bulgarian National Scientific Fund under Grant KP-06-H22/4/04.12.2018.\\		
			The first author is also supported by L'Oreal and UNESCO fellowship 'For Women in Science' 2018.}
		\author{Mira Bivas \fnref{label2,label3}}
		\author{Mikhail Krastanov\fnref{label2,label3}}
		\author{Nadezhda Ribarska\corref{cor1}\fnref{label2,label3}}
		\cortext[cor1]{Corresponding author: ribarska@fmi.uni-sofia.bg}
		

		\address[label2]{Faculty of Mathematics and Informatics,  Sofia University, James Bourchier Boul. 5,  1126 Sofia, Bulgaria}
		\address[label3]{Institute of
			Mathematics and Informatics, Bulgarian Academy of Sciences, G.Bonchev str., bl. 8, 1113 Sofia, Bulgaria}

		\begin{abstract}
 	This is the first of two closely related papers on transversality. Here we introduce the notion of tangential transversality of two closed subsets of a Banach space.  It is an intermediate property between transversality and subtransversality. Using it, we obtain a variety of known results and some new ones in a unified way. Our proofs do not use variational principles and we concentrate mainly on tangential conditions in the primal space.
			
		\end{abstract}


\end{frontmatter}


\section{Introduction}

Transversality is a classical concept of mathematical analysis and differential topology. Recently, it has proven to be useful in variational analysis as well. As it is stated in \cite{Ioffe}, the transversality-oriented language is extremely natural and convenient in some parts of variational analysis, including subdifferential calculus and nonsmooth optimization.

The classical definition of transversality at an intersection point of two smooth manifolds
in a Euclidean space is that the sum of the corresponding tangent spaces at the intersection point
is the whole space (cf. \cite{DG1}, \cite{DG2}). Equivalently, the intersection of the corresponding normal spaces is the origin.

In order to prove the Pontryagin maximum principle (cf., for example, the bibliography of \cite{HS}), Hector Sussmann generalizes the definition of transversality for closed convex cones in $\mathbb{R}^n$:   cones $C^A$ and $C^B$ are
transversal if and only if $$C^A - C^B = \mathbb{R}^n$$ and strongly transversal, if they are transversal and $C^A \cap C^B \not
= \{0\}$ (cf. Definitions 3.1  and  3.2 from \cite{HS}). In finite-dimensional case, strong transversality of the approximating cones of the same type (either Clarke or Boltyanski) is a sufficient condition for local nonseparation of sets. The sets $A$, $B$
containing a point $x_0$ are said to be   locally  separated at $x_0$, if there exists a neighborhood $\Omega$ of $x_0$ so
that $\Omega\cap A \cap B = \{x_0\}$ (cf., for example, the introduction of \cite{HS}).  In infinite-dimensional case, strong transversality of the approximating cones of the same type  does not imply  local nonseparation of sets -- take for example a Hilbert cube	$A:=\{(x_n)\in l_2: \ |x_n|\le 1/n\}\subset l_2$ and a ray $B:= \{\lambda y: \ \lambda \ge 0\}$, where $y:=(1/n^{3/4})_{n=1}^\infty$. We have that the corresponding Clarke tangent cones $\hat T_A(\mathbf{0})= l_2$ and $\hat T_B(\mathbf{0})= B$ are strongly transversal, while the sets $A$ and $B$ are locally separated at 0.

In the literature there exist many notions generalizing the classical transversality as well as transversality of cones. Some of them are introduced under different names by different authors, but actually coincide. We refer to \cite{Kruger2018} for a survey of terminology and comparison of the available concepts. The central ones among them are \textit{transversality} and \textit{subtransversality}. They are also objects of study in the recent book \cite{IoffeBook}.

These notions are closely connected to some important relations between tangent [normal] cones to two sets and the tangent [normal] cone to their intersection. To be more specific, let $T_C(x)$ be the tangent cone (in some sense -- Bouligand, derivable, Clarke, ...) to a closed subset $C$ of the Banach space $X$ at $x\in C$ and $N_C(x)$ be the normal cone (in some sense -- proximal, limiting, $G$-normal, Clarke, ...) to a closed set $C\subset X$ at $x\in C$. For the sake of convenience, we introduce the following
\begin{definition}
	Let $A$ and $B$ be closed subsets of the Banach space $X$ and let $x_0$ belong to $A\cap B$.	
	We say that $A$ and $B$ have tangential intersection property at $x_0$ with respect to the type(s) of the approximating cones $T_A(x_0)$ and $T_B(x_0)$ if
	$$T_{A\cap B}(x_0) \supset T_A(x_0)\cap T_B(x_0) \, . $$
	
	We say that $A$ and $B$ have normal intersection property at $x_0$ with respect to the type(s) of the normal cones $N_A(x_0)$ and $N_B(x_0)$ if
	$$N_{A\cap B}(x_0) \subset N_A(x_0) + N_B(x_0)$$
		and the right-hand side of the above inclusion is weak$^*$-closed.
\end{definition}

The reader is referred to \cite{IoffeBook} and \cite{Penot} for the precise definitions of the above mentioned cones.

Equivalent definitions of transversality (see Theorem 2 in \cite{Kruger2018}) have been around for almost 20 years and are mostly used as   sufficient conditions for normal intersection property with respect to the limiting normal cones in Asplund spaces (cf. \cite{Mord}, \cite{Penot}).
The term subtransversality is recently introduced in \cite{DIL2015} in relation to proving linear convergence of the alternating projections algorithm. However, this property has been around for more than 20 years as well, but under different names -- see Remark 4 in \cite{Kruger2018} and the references therein. It is a key assumption for two types of results: linear convergence of sequences generated by projection algorithms and a qualification condition for normal intersection property with respect to the  limiting normal cones and a sum rule for the limiting subdifferentials. Subtransversality is a weaker condition than transversality, but also implies normal intersection property  (cf.  Theorem 6.41 in \cite{Penot} for the limiting normal cones in Asplund spaces  and Theorem 7.13 in \cite{IoffeBook} for the $G$-normal cones in Banach spaces).

We arrived to the study of transversality of sets when investigating Pontryagin's type maximum principle for  optimal control problems with terminal
constraints in infinite dimensional state space. In order to prove a nonseparation result (if one can not separate the approximating cones of two closed sets  at a common point  and, moreover, the cones have nontrivial intersection, then the sets can not be separated as well) we introduced the notion of uniform tangent set (cf. \cite{KR17}). It happened to be very useful for obtaining necessary conditions for  optimal control problems in infinite dimensional state space, because the diffuse variations (which are naturally defined and easy to calculate) form a uniform tangent set to the reachable set of a control system. The present manuscript and \cite{BKRstt} are  an effort to understand the relation of our study to the established results and methods of nonsmooth optimization. As we arrived to some known notions and results using our approach, the proofs of the known theorems are completely different from the classical ones and, moreover, we found some new   results. Our proofs do not use variational principles and we concentrate mainly on tangential conditions in the primal space. We were able to obtain a vast variety of results in a unified and economical way.

Here we introduce the notion of \textit{tangential transversality} in Banach spaces. It is an intermediate property between transversality and subtransversality.      There are many really useful sufficient conditions for tangential transversality of two sets. One of them -- strong tangential transversality -- involves uniform tangent sets and is studied in detail in \cite{BKRstt}. Moreover, we obtain here a sufficient condition for tangential transversality which is different from  the conditions for subtransversality we know about.

The present  paper is organized as follows: The second section contains the definition of tangential transversality  and the main technical tool allowing us to use this concept. Its proof uses a   nontrivial
construction which essentially appeared in \cite{KRTs} and \cite{KR17}. The relation of tangential transversality with the concepts of  transversality and subtransversality is obtained.  Some consequences of subtransversality are gathered in the third and the fourth sections.  These include a nonseparation theorem, an abstract Lagrange multiplier rule and some tangential intersection properties.   In the fifth section  a sufficient condition for tangential transversality is proved when one of the sets involved is  massive. As corollaries a   sum rule for $G$-subdifferentials and a   Lagrange multiplier rule are obtained.

Throughout the paper if $Y$ is a Banach space, we will denote by ${\mbox{\bf B}_Y}$ [$ \bar{\mbox{\bf{B}}}_Y$] its open [closed] unit ball, centered at the origin. The index could be omitted if there is no ambiguity about the space. If $S$ is a closed subset of $Y$ at $y\in S$, we will denote by $T_S(y)$ the Bouligand tangent cone to $S$ at $y$, i.e.
$$
T_S(y) := \left\{
v \in Y: \displaystyle \frac{  y_k -y }{\tau _k }    \to v \begin{array}{ll}
 & \mbox{ for some sequences } y_k \in S, y_k \to y     \\
  &  \mbox{ and } \tau_k >0, \tau_k \to 0
\end{array}
\right\};
$$
\noindent by $G_S(y)$ the derivable tangent cone to $S$ at $y$, i.e.
$$
G_S(y) := \left\{
v \in Y: \displaystyle \frac{   \xi(\tau_k)  -y }{\tau_k}    \to v  \begin{array}{ll}
 & \mbox{ for some vector-valued function} \\
&  \xi: [0,\varepsilon]  \to S,  \xi (0)=y
     \mbox{ and for every} \\ & \mbox{choice of a sequence } \tau_k >0, \tau_k \to 0
\end{array}
\right\};
$$
\noindent
 and by $\hat T_S(y)$ the Clarke tangent cone to $S$ at $y$, i.e.
 $$
\hat T_S(y) := \left\{
v \in Y:   \begin{array}{l }
   \mbox{ for each sequence } y_k \in S, y_k \to y  \mbox{ and }   \\
    \mbox{ for each sequence } \tau_k >0, \tau_k \to 0 \mbox{ there }\\
    \mbox{  exists a sequence } z_k \in S    \mbox{ with  } \displaystyle \frac{  z_k -y_k }{\tau _k }    \to v
\end{array}
\right\}.
$$

\textbf{Acknowledgment.} We are grateful to Prof. H. Frankowska for a fruitful discussion which has led us to the present formulation of the Lagrange multiplier rule (Theorem \ref{Lagrange}); to Prof. A. Ioffe for his useful comments and suggestions and to the unknown referee for his/her careful
reading of the manuscript and for numerous remarks that helped us to improve it.

\section{Tangential transversality}

The definition below is central for our considerations:

\begin{definition}
	Let $A$ and $B$ be closed subsets of the Banach space $X$.
	We say that $A$ and $B$ are tangentially transversal at $x_0 \in A \cap B$, if there exist $M>0$, $\delta >0$ and  $\eta >0$ such that for any
two different points $x^A \in (x_0 + \delta \bar {\mbox{\bf B}})\cap A$ and $x^B \in (x_0 + \delta \bar {\mbox{\bf B}})\cap B$, there exists a sequence $\{t_m\}$, $t_m\searrow 0$, such that for every $m \in \mathbb{N}$ there exist $w^A_m \in X$ with $\|w^A_m\|\le M$ and $x^A + t_m w^A_m \in A$, and $w^B_m \in X$ with $\|w^B_m\|\le M$, $x^B + t_m w^B_m \in B$, and the following inequality holds true $$\|x^A-x^B + t_m(w^A_m-w^B_m)\| \le \|x^A-x^B\| - t_m \eta\, .$$
\end{definition}

It is hard to find the  intuition behind this notion. Roughly speaking, the meaning  is that for each two points in a neighborhood of $x_0$ the distance between them can be decreased at a linear rate. Different sufficient conditions and examples of applications of this notion can be found in \cite{BKRstt}. Another
way of looking at this property, due Professor Alexander Ioffe, is presented in the following proposition.
It highlights the metric nature of the concept in question (a comment of Prof. A. Ioffe).

\begin{prop}
Let $A$ and $B$ be closed subsets of the Banach space $X$.
   The sets $A$ and $B$ are tangentially transversal at $x_0 \in A \cap B$ if and only if there exist $\delta >0$ and  $\zeta >0$ such that for any
two different points $x^A \in (x_0 + \delta \bar {\mbox{\bf B}})\cap A$ and $x^B \in (x_0 + \delta \bar {\mbox{\bf B}})\cap B$, there exists a sequence $\{s_m\}$, $s_m\searrow 0$, such that for every $m \in \mathbb{N}$   the following inequality holds true
$$\mathrm{dist} \, \left( \bar {\mbox{\bf{B}}}_{s_m} (x^A)\cap A, \bar {\mbox{\bf{B}}}_{s_m} (x^B)\cap B\right) \le \|x^A-x^B\| - s_m \zeta\, .$$
Here $\bar {\mbox{\bf{B}}}_{s_m} (x)$ is the set $x+s_m \bar {\mbox{\bf{B}}}$, $\mathrm{dist} \, (C,D):= \inf \{ \| x-y\| : x\in C, y\in D\}$.
\end{prop}

\begin{proof}
Let $A$ and $B$ be tangentially transversal at $x_0 \in A \cap B$ with constants $M>0$, $\delta >0$ and  $\eta >0$. Put $\zeta := \eta/M$ and choose two arbitrary different points $x^A \in (x_0 + \delta \bar {\mbox{\bf B}})\cap A$ and $x^B \in (x_0 + \delta \bar {\mbox{\bf B}})\cap B$.
We set $s_m:=Mt_m$ and note that
$$\mathrm{dist} \, \left( \bar {\mbox{\bf{B}}}_{s_m} (x^A)\cap A, \bar {\mbox{\bf{B}}}_{s_m} (x^B)\cap B\right) \le $$ $$\le \left\| (x^A + t_m w^A_m)-
(x^B + t_m w^B_m)\right\| \le \|x^A-x^B\| - t_m \eta = \|x^A-x^B\| - s_m \zeta \ .$$
Let now $A$ and $B$ satisfy the metric condition in the formulation of this proposition. Choose any $\eta$ with $0<\eta <\zeta$. Then
$$\mathrm{dist} \, \left( \bar {\mbox{\bf{B}}}_{s_m} (x^A)\cap A, \bar {\mbox{\bf{B}}}_{s_m} (x^B)\cap B\right) \le \|x^A-x^B\| - s_m \zeta <\|x^A-x^B\| - s_m \eta$$
and therefore there exist $x^A + s_m w^A_m\in \bar {\mbox{\bf{B}}}_{s_m} (x^A)\cap A$  and $x^B + s_m w^B_m\in \bar {\mbox{\bf{B}}}_{s_m} (x^B)\cap B$ such that
$$\left\| (x^A + s_m w^A_m)-
(x^B + s_m w^B_m)\right\| <\|x^A-x^B\| - s_m \eta \ .$$
As $\|w_m^A\|\le 1$ and $\|w_m^B\|\le 1$, the definition of tangential transversality is satisfied with $M:=1$, $\delta >0$,  $\eta >0$ and $t_m:=s_m$.
\end{proof}

The following theorem is the main technical result to be used later on. The idea of its proof is already present in the proofs of Theorem 3.3 in \cite{KRTs} and Theorem 2.6 in \cite{KR17}.

\begin{thm}\label{ttth}
	Let the closed sets $A$ and $B$ be tangentially transversal at $x_0 \in A \cap B$ with constants $M>0$, $\delta >0$ and  $\eta >0$. Let $x^A\in A$ and $x^B\in B$ be such that
	\begin{equation} \label{init}
	\max \left\{ \|x^A-x_0\|, \|x^B-x_0\| \right\} + \frac{M}{\eta} \|x^A-x^B\| \le \delta \, .
	\end{equation}
	Then, there exists $x^{AB} \in A \cap B$ with
	$\|x^{AB} - x^A\| \le \frac{M}{\eta} \|x^A-x^B\| \mbox{ and } \|x^{AB} - x^B\| \le \frac{M}{\eta} \|x^A-x^B\| \, .$
\end{thm}

\begin{proof}

	We are going to construct inductively four transfinite sequences
	indexed by ordinal numbers (cf., for example,  $\S$ 2 Ordinal
	numbers of Chapter 1 in \cite{Jech}). More precisely, we prove that
	there exist an ordinal number $\alpha_0$ and transfinite sequences
	$\{x_\alpha^A\}_{1 \le \alpha \le \alpha_0}\subset (x_0 + \delta \bar {\mbox{\bf B}})\cap A$,
	$\{x_\alpha^B\}_{1 \le \alpha \le \alpha_0}\subset (x_0 + \delta \bar {\mbox{\bf B}})\cap B$,
	$\{t_\alpha\}_{1 \le \alpha \le \alpha_0}\subset [0, +\infty)$,
	$\{h_\alpha\}_{1 \le \alpha <\alpha_0}\subset (0,+\infty)$
	such that $x_{\alpha_0}^A =x_{\alpha_0}^B$ and  for each $\alpha \in [1, \alpha_0 ]$
	we have that $\displaystyle
	t_\alpha=\sum_{1\le\beta<\alpha} h_\beta$
	and   the
	following estimates hold true for each $\beta$, $1 \le \beta \le \alpha$ and each $\gamma$, $1 \le \gamma\le \alpha$:
	
	\begin{enumerate} \item [(S1)]  $\|x_\beta^A - x_\beta^B\| \le \|x_1^A - x_1^B\|  - t_\beta \eta$ (and hence $t_\beta$  is bounded by $\frac{\|x_1^A - x_1^B\|}{\eta}$);\\
		\item [(S2)]  $ \|x_\beta^A     - x_0\| \le \|x_1^A     -  x_0\|+t_\beta M$;\\
		\item [(S3)]  $ \|x_\beta^B- x_0\| \le \|x_1^B      -  x_0\|+t_\beta M$;\\
		\item [(S4)]  $\|x_\beta^A - x_\gamma^A\| \le M \left(t_\beta - t_\gamma\right)$;\\
		\item [(S5)]  $\|x_\beta^B - x_\gamma^B\| \le M \left(t_\beta - t_\gamma\right)$.\\
	\end{enumerate}

	We implement our construction using induction on $\alpha$.
We start with  $x_1^A := x^A\in (x_0 + \delta \bar {\mbox{\bf B}})\cap A$, $x_1^B :=x^B\in (x_0 + \delta \bar {\mbox{\bf B}})\cap B$ and $t_1= 0$.
If $x^A_1=x _1^B$, we set $\alpha_0:=1$ and terminate the process.
	If $x^A_1\not=x _1^B$, we set $h_1$ to be equal to the first element of the sequense $\{t_m\}$ from the definition of tangential transversality.
	It is straightforward to verify     the induction
	assumptions (S1)-(S5)   for $\beta = 1$ and $\gamma = 1$.
	
	Assume that $x_\beta^A\in (x_0 + \delta \bar {\mbox{\bf B}})\cap A$, $x_\beta^B\in (x_0 + \delta \bar {\mbox{\bf B}})\cap B$, $h_\beta >0$ and
	$t_\beta = \sum_{\gamma<\beta} h_\gamma>0$ are constructed and (S1)-(S5) are true for all
	ordinals $\beta$  less than $\alpha$ and the process has not been terminated.
	
	Let us first consider the case when $\alpha$  is a non limit
	ordinal number, i.e. $\alpha = \beta+ 1$. As $\beta <\alpha_0$
	(the process has not been terminated), we have $\|x_\beta^A -
	x_\beta^B \|\not = 0$. Then we set
	$
	h_\beta \in (0, \|x_\beta^A -x_\beta^B\| ]
	$
	to be equal to $t_m$ for some $m$, where the sequence $\{t_m\}$ is from the definition of tangential transversality (it is possible, because $t_m \searrow 0$).
	Then, using again the definition of tangential transversality, there exist $w^A_\beta \in X$ with $\|w^A_\beta\|\le M$ and $w^B_\beta \in X$ with $\|w^B_\beta\|\le M$
	such that
	$$x_\alpha^A := x_\beta^A + h_\beta w^A_\beta \in A \ , $$
	$$x_\alpha^B := x_\beta^B + h_\beta w^B_\beta \in B$$
	and
	\begin{align*}
	\|x_\alpha^A -x_\alpha^B \| &= \|x_\beta^A-x_\beta^B + h_\beta(w^A_\beta-w^B_\beta)\| \le \|x_\beta^A-x_\beta^B\| - h_\beta \eta \\
	&\le \|x_1^A-x_1^B\| -(t_\beta+ h_\beta) \eta  \, .
	\end{align*}
	Setting $t_\alpha:= t_\beta + h_\beta $, we have
	\begin{align*}
	\|x_\alpha^A -x_\alpha^B \| \le  \|x_1^A-x_1^B\| - t_\alpha \eta \, .
	\end{align*}
	Therefore, (S1) is verified for $\alpha$.

	(S2) yields
	\begin{align*}\|x_\alpha^A - x_0\| \le&  \|x_\beta^A - x_0\| +h_\beta \|w^A_\beta\| \\
	\leq &  \|x_1^A - x_0\| +t_\beta M + h_\beta M = \|x_1^A - x_0\| +t_\alpha M \ .
	\end{align*}

	Analogously, using (S3) instead of (S2), we obtain
	$$\|x_\alpha^B -    x_0\| <  \|x_1^A - x_0\| +t_\alpha M \ .$$
	
	Using the estimate for $\|x_1^A-x_0\|$ from \eqref{init} and that $t_\beta \le \frac{\|x_1^A - x_1^B\|}{\eta}$, we obtain
	\begin{align*}\|x_\alpha^A - x_0\| \le&   \|x_1^A - x_0\| +t_\alpha M \le \delta - \frac{M}{\eta}\|x_1^A - x_1^B\| + \frac{\|x_1^A - x_1^B\|}{\eta} M = \delta \ ,
	\end{align*}
	and similarly
	$$ \|x_\alpha^A - x_0\| \le \delta \ .$$
	
	Now  let $\gamma<\alpha$. Then
	\begin{align*}
	\|x_\alpha^A& - x_\gamma^A\| = \|x_\beta^A-x_\gamma^A+
	h_\beta v_\beta^A \|\le \|x_\beta^A-x_\gamma^A\| +
	h_\beta\| v_\beta^A \| \\
	&\le M(t_\beta-t_\gamma)+M(t_\alpha-t_\beta)=M(t_\alpha-t_\gamma)
	\end{align*}
	and in the same way
	$$ \|x_\alpha^B - x_\gamma^B \| \le M(t_\alpha-t_\gamma) \, . $$
	
	We have verified the inductive assumptions (S1)-(S5) for the case of a non limit ordinal number $\alpha$.
If $x^A_\alpha=x _\alpha^B$, we set $\alpha_0:=\alpha$ and terminate the process.
	
	We next consider the case when $\alpha$ is a limit  ordinal
	number. Let $\beta<\alpha$ be arbitrary. Then $\beta+1<\alpha$
	too. Since the transfinite process has not stopped at   $\beta+1$,
	then $\|x_\beta^B - x_\beta^A\|>0$, and hence taking into account
	(S1) we obtain that
	$$
	t_\beta< \frac{\|x_1^A - x_1^B\|}{\eta} \ .
	$$
	Hence the increasing transfinite sequence
	$\{t_\beta\}_{1\le\beta<\alpha}$ is bounded, and so it is
	convergent. We denote $t_\alpha:=lim_{\beta \to \alpha} t_\beta= lim_{\beta \to \alpha} \sum_{\gamma < \beta} h_\gamma= \sum_{\gamma < \alpha} h_\gamma$. Since $\|x_\beta^A - x_\gamma^A \| \le
	(t_\beta-t_\gamma)M$, the transfinite sequence
	$\{x_\beta^A\}_{1\le\beta<\alpha}$ is fundamental. Hence there
	exists $x_\alpha^A$ so that $\{x_\beta^A\}_{1\le\beta<\alpha}$ tends
	to $x_\alpha^A$ as $\beta$ tends to $ \alpha$ with $\beta<\alpha$. In the same way one can
	prove the existence of $x_\alpha^B$ so that the transfinite
	sequence $\{x_\beta^B\}_{1\le\beta<\alpha}$ tends to $x_\alpha^B$ as
	$\beta$ tends to $\alpha$. To verify the inductive assumptions for
	$\alpha$, one can just take a limit for $\beta$ tending to $\alpha$ with $\beta<\alpha$ in
	the same assumptions written for each  $\beta <\alpha$. If $x^A_\alpha=x _\alpha^B$, we set $\alpha_0:=\alpha$ and terminate the process.
	
	We have   constructed inductively the transfinite sequences $$\{ x_\beta^A \}_{\beta \le \alpha}\subset A,
	\ \{ x_\beta^B \}_{\beta \le \alpha}\subset B$$ and $\{ t_\beta\}_{\beta
		\le \alpha}\subset [0,+\infty)$. The process   terminates when
	$x_\alpha^A=x_\alpha^B$ for some $\alpha$. Since
	$$
	\|x_\alpha^A - x_\alpha^B \| \le \|x_1^A - x_1^B\|  - t_\beta \eta
	$$
	and the transfinite sequence $t_\alpha$ is strictly increasing, the  equality
	$x_\alpha^A=x_\alpha^B$ will be satisfied for some
	$\alpha=\alpha_0$ strictly preceding the first uncountable ordinal
	number. Indeed, the successor ordinals indexing the
	so constructed transfinite sequences form a countable set (because to every successor
	ordinal $\alpha+1$ corresponds the open interval $(t_\alpha, t_\alpha+h_\alpha)\subset \mathbb{R}$,
	these intervals are disjoint and the rational numbers are countably many and
	dense in $\mathbb{R}$). Therefore, $\alpha_0$ is countable accessible.
	On the other hand, according to the Corollary after Lemma 5.1 on page 40
	of \cite{Jech},  $\aleph_{\gamma+1} $ is a regular cardinal
	(under the assumption of the Axiom of choice) for every
	$\gamma$, in particular the first uncountable cardinal $\aleph_1$ is not countably accessible.
	Thus $\omega_1$ is not countably accessible (as $\omega_1$     is the first ordinal
	with $|\omega_1|=\aleph_1$).
	Hence our inductive process must terminate before $\omega_1$.

	Then $x^{AB}:=x_{\alpha_0}^A=x_{\alpha_0}^B \in A \cap B$ and because of (S1)
	we have that
	$$t_{\alpha_0} \le \frac{\|x_1^A - x_1^B\|}{\eta} \, .$$
	Applying (S4)  we obtain
	\begin{align*}
	\|x^{AB}& - x_1^A\| \leq M(t_{\alpha_0}-t_1) \leq  \frac{M}{\eta} \|x^A-x^B\| \, .
	\end{align*}
	Analogously, due to (S5),
	$$  \|x^{AB} - x_1^B\| \leq  \frac{M}{\eta} \|x^A-x^B\| \, .$$
	
	This completes the proof.
\end{proof}

\begin{prop}
  \label{cond}
	If the sets $A$ and $B$ are tangentially transversal at $x_0 \in A \cap B$ with constants
	$M>0$, $\delta >0$ and  $\eta >0$, then we have that \eqref{init} holds true for all
	$x^A\in A$, $\|x^A-x_0\| \le \zeta$ and  $x^B\in B$, $\|x^B-x_0\| \le \zeta$, where
	$\zeta :=\displaystyle\frac{\delta}{1+2\frac{M}{\eta}}$.
\end{prop}
\begin{proof}
  Indeed,
	\begin{align*}
	\max& \left\{ \|x^A-x_0\|, \|x^B-x_0\| \right\} + \frac{M}{\eta} \|x^A-x^B\| \\
	&\le \zeta + \frac{M}{\eta} \left(  \|x^A-x_0\| + \|x_0-x^B\| \right)
	\le \zeta + \frac{M}{\eta}2\zeta = \delta \, .
	\end{align*}
\end{proof}

We are going to show that transversality implies tangential transversality, which implies subtransversality due to the above theorem. The definitions below are taken from the recent book \cite{IoffeBook}.

\begin{definition}
	Let $A$ and $B$ be closed subsets of the Banach space $X$.
	$A$ and $B$ are said to be transversal at $x_0 \in A \cap B$, if there exist $\delta >0$ and $K>0$, such that
	$$d(x, (A-a)\cap(B-b)) \leq K(d(x, A-a) + d(x, B-b)) $$
	for all $x \in x_0 + \delta \bar {\mbox{\bf B}} $ and $a$ and $b$ close enough to the origin.
\end{definition}

\begin{definition}
	Let $A$ and $B$ be closed subsets of the Banach space $X$.
	$A$ and $B$ are said to be subtransversal at $x_0 \in A \cap B$, if there exist $\delta >0$ and $K>0$, such that
	$$d(x, A\cap B) \leq K(d(x, A) + d(x, B)) $$
	for all $x \in x_0 + \delta \bar {\mbox{\bf B}} $.
\end{definition}

\begin{prop}
	Let the closed sets $A$ and $B$ be transversal at $x_0 \in A \cap B$. Then, $A$ and $B$ are tangentially transversal at $x_0$.
\end{prop}

\begin{proof}
	In the proof we are going to use the equivalent definition of transversality given in \cite{Kru15} (cf. Definition 3.1 (iii) and Theorem 3.1 (iii) in \cite{Kru15}):\\
	$A$ and $B$ are transversal at $x_0 \in A \cap B$, if and only if there exist $\alpha>0$ and $\delta >0$ such that
	\begin{equation}\label{kru}
	(A-x^A-\rho w_1) \cap (B-x^B-\rho w_2) \cap \rho \bar {\mbox{\bf B}} \neq \emptyset
	\end{equation}
	for all $\rho \in (0, \delta)$, $w_i \in \alpha \bar {\mbox{\bf B}}, \, i=1,2$,  $x^A \in (x_0 + \delta \bar {\mbox{\bf B}})\cap A$ and $x^B \in (x_0 + \delta \bar {\mbox{\bf B}})\cap B$.
	
	We will show that $A$ and $B$ are tangentially transversal at $x_0$ with constants $M:= \alpha +1$, $\delta$ and  $\eta:= {\alpha}$.
	
	Let us fix $x^A \in A\cap (x_0 + \delta \bar {\mbox{\bf B}})$, $x^B \in B\cap (x_0 + \delta \bar {\mbox{\bf B}})$ with $x^A\not = x_B$ and
let $t_m \in (0, \min \{\delta, \frac{\|x^B-x^A\|}{\alpha} \})$. We put $w_1:= \alpha \frac{x^B-x^A}{\|x^B-x^A\|} \in \alpha \bar {\mbox{\bf B}} $ and $w_2 := \mathbf{0}$. Then, \eqref{kru} (with $\rho:=t_m$) is equivalent to the existence of $u \in \bar {\mbox{\bf B}}$ such that
	$$t_m u \in (A-x^A-t_m w_1)\cap(B-x^B-t_m w_2) \, . $$
	The last inclusion implies that
	$$x^A+t_m w^A_m \in A \mbox{ and } x^B+t_m w^B_m \in B \, , $$
	where $w^A_m := w_1 + u$ and $w^B_m := w_2 + u$. We also have that
	$\|w^A_m\| \le \alpha+1 = M$ and $\|w^B_m\| \le 1\le  M$.
	
	We estimate
	\begin{align*}
	\|x^A&-x^B + t_m(w^A_m-w^B_m)\| = \Big\|x^A-x^B + t_m \alpha \frac{x^B-x^A}{\|x^B-x^A\|}\Big\| \\
	&=  \|x^A-x^B\|\left|1-\frac{t_m \alpha}{\|x^B-x^A\|}\right| = \|x^A-x^B\| - t_m \eta \, .
	\end{align*}
	This proves the tangential transversality.
	
\end{proof}

\begin{prop}\label{subtr}
	Let the closed sets $A$ and $B$ be tangentially transversal at $x_0 \in A \cap B$. Then, $A$ and $B$ are subtransversal at $x_0$.
\end{prop}

\begin{proof}
	Let the constants $M>0$, $\delta >0$ and  $\eta >0$ be from the definition of tangential transversality.
	
	Let us set $\zeta
	:=\displaystyle\frac{\delta}{2\left(1+ 2\frac{M}{\eta}\right)}\in (0, \delta)$.
	Let $x$ be an arbitrary element of $x_0+ \zeta {\mbox{\bf B}}$.
	Let us fix an arbitrary $\varepsilon \in (0, \zeta - \|x-x_0\|)$.
	We have that there exist $x^A \in  A$ and $x^B \in B$ such that
	\begin{equation}\label{m1}\|x^A-x\| < d(x, A) + \varepsilon \mbox{ and } \|x^B-x\| <
	d(x, B) + \varepsilon \  . \end{equation} Since $d(x, A)  \le
	\|x-x_0\|$, we obtain that
	$$ \|x^A-x\| < \|x_0-x\| + \varepsilon < \zeta $$
	and therefore $x^A \in (x_0+2\zeta {\mbox{\bf B}})\cap A$. Analogously, $x^B \in (x_0+2\zeta {\mbox{\bf B}})\cap B$.
	
	We have that
	$$      \max \left\{ \|x^A-x_0\|, \|x^B-x_0\| \right\} + \frac{M}{\eta} \|x^A-x^B\| < 2\zeta + 4\frac{M}{\eta}\zeta = \delta \, . $$
	We can apply Theorem \ref{ttth} and obtain $x^{AB} \in A \cap B$ with
	\begin{equation}\label{m2}\|x^{AB} - x^A\| \le \frac{M}{\eta} \|x^A-x^B\|
	\mbox{ and } \|x^{AB} - x^B\| \le \frac{M}{\eta} \|x^A-x^B\|
	.\end{equation}
	
	Applying (\ref{m1}) and (\ref{m2}), we obtain
	\begin{align*}
	d(x, A\cap B) &\le \|x-x^{AB}\| \le \|x-x^A\| + \|x^A-x^{AB}\| \\
	<& d(x, A) + \varepsilon + \frac{M}{\eta} \|x^A-x^B\| \le d(x, A) + \varepsilon + \frac{M}{\eta}(\|x^A-x\|+\|x-x^B\|) \\
	<&  d(x, A) + \varepsilon + \frac{M}{\eta}\left(d(x, A) + \varepsilon + d(x, B) + \varepsilon\right)\\
	\le& \left(1+\frac{M}{\eta}\right)(d(x, A)  + d(x, B)) + \varepsilon\left(1+ 2 \frac{M}{\eta} \right) \ .
	\end{align*}
	
	Letting $\varepsilon$ go to $0$ proves the subtransversality with constants $\zeta>0$ and $K := 1+\frac{M}{\eta} >0$.
\end{proof}

\section{A Lagrange multiplier rule}

It is our understanding that the following result is crucial for obtaining necessary optimality conditions.

\begin{prop}[Nonseparation result]\label{nonsep}
	Let $A$ and $B$ be closed subsets of the Banach space $X$. Let $A$ and $B$ be subtransversal at $x_0 \in A \cap B$ with constants $\delta>0$ and  $K >0$. Let  there exist $v^A$ with unit norm which belongs to the Bouligand tangent cone to $A$ at $x_0$, $v^B$  with unit norm which belongs to the derivable  tangent cone to $B$ at $x_0$ and let $\|v^A-v^B\|< \frac{1}{K}$. Then $A$ and $B$ cannot be locally  separated at $x_0$.
\end{prop}

\begin{proof}
	Since  $v^A$ belongs to the Bouligand tangent cone to $A$ at $x_0$, we have that there exist sequences $t_m \searrow 0$ and $v_m^A \to v^A$ such that
	$$x_m^A := x_0 +t_m v_m^A \in A \, . $$
	Since $v^B$  belongs to the derivable  tangent cone to $B$ at $x_0$, we have that for all small enough $t>0$ there exists $v_t^B \in X$, such that $x_0 +t v_t^B \in B$ and $v_t^B \to v^B$ as ${t\searrow 0}$. Let us set
	$$x_m^B := x_0 +t_m v_m^B \in B$$
for $m$ -- large enough. We wrote $v_m^B$ instead of $v_{t_m}^B$ for the sake of simplicity.
	
	From the triangle inequality we obtain that $\|v_m^A\|\ge \|v^A\|-\|v_m^A-v^A\| = 1 - \|v_m^A-v^A\|$ and therefore
for $ m $   large enough we have
	$$t_m = \frac{\|x_m^A-x_0\|}{\|v_m^A\|} \le  \frac{\|x_m^A-x_0\|}{1 - \|v_m^A-v^A\|} \, $$
	and
	\begin{align*}
	\|x_m^A-x_m^B\| = t_m \|v_m^A-v_m^B\| \le \frac{\|x_m^A-x_0\|}{1 - \|v_m^A-v^A\|} \left( \|v_m^A-v^A\|+ \|v^A-v^B\| + \|v^B-v_m^B\|\right) \, .
	\end{align*}
	
	We have that $$\frac{\|v_m^A-v^A\|+ \|v^A-v^B\| + \|v^B-v_m^B\|}{1 - \|v_m^A-v^A\|} \to_{m \to +\infty} \|v^A-v^B\|$$ and $\|v^A-v^B\|< \frac{1}{K+\varepsilon}$ for some small enough $\varepsilon>0$. Therefore there exists $m_0\in \mathbb{N}$ such that
	\begin{equation}\label{est}
	\|x_m^A-x_m^B\| \le \|x_m^A-x_0\| \cdot \frac{1}{K+\varepsilon}
	\end{equation}
	for all $m\ge m_0$.
	
	Let $m_1\ge m_0$ be such that $t_{m}\|v_m^A\| \le \delta$ and $t_{m}\|v_m^B\| \le \delta$ whenever $m\ge m_1$. Then, for  $m\ge m_1$ we have
	$$d\left( x^A_m, A\cap B\right) \le K\left( d\left( x^A_m, A\right)+d\left( x^A_m, B\right)\right) \le K\cdot d\left( x^A_m, x^B_m\right) = K \left\| x^A_m -x^B_m\right\|\, . $$
	From the definition of a distance from a point to a set there exists $x^{AB}_m \in A \cap B$ with
	$$\left\|x^{AB}_m - x^A_m\right\| \le d\left( x^A_m, A\cap B\right) + \frac{\varepsilon}{2}\left\| x^A_m -x^B_m\right\| \ .$$
Note that if $x^A_m=x^B_m$ we just put $x^{AB}_m$ to coincide with these points and all addends are zero.
Then
$$\left\|x^{AB}_m - x^A_m\right\| \le K \left\| x^A_m -x^B_m\right\| + \frac{\varepsilon}{2}\left\| x^A_m -x^B_m\right\| \le  \left\|x_m^A-x_0\right\| \frac{K+\varepsilon/2}{K+\varepsilon} < \left\|x_m^A-x_0\right\|  \, $$
	using \eqref{est}.
	Therefore $x^{AB}_m \neq x_0$. Moreover,
	\begin{align*}
	& \|     x^{AB}_m - x_0 \|\le  \|     x^{A }_m - x_0 \| + \|   x^A_m - x^{AB}_m \| \le \\ & \le 2 \|     x^{A }_m - x_0 \| \le 2   t_m \|v_m^A\|
	    \longrightarrow_{m\to +\infty} 0 \, .
	\end{align*}
	Thus, $x^{AB}_m \to x_0$ and $A$ and $B$ cannot be locally  separated at $x_0$.
\end{proof}

We will apply the above nonseparation result to obtain an abstract Lagrange multiplier rule. Let $X$ be a Banach space. We consider $X\times \mathbb{R}$ equipped with the uniform norm $\|(x, r)\| := \max\{\|x\|, |r| \}$. We will need  Lemma \ref{nonsepcones} below. It is a natural generalisation of the fact that in finite dimensions if  two cones are transversal and one of them is not a subspace, then they are strongly transversal.

\begin{lem}\label{nonsepcones}
	Let $\tilde C_1$ and $\tilde C_2 := C_2 \times (-\infty, 0]$ be closed convex cones in $X\times \mathbb{R}$ (hence $C_2$ is a closed convex cone in $X$).
	Let $\tilde C_1 - \tilde C_2$ be dense in $X\times \mathbb{R}$.
	Then, for each $\varepsilon>0$ there exist $\tilde w_1 \in \tilde C_1$ and $\tilde w_2 \in \tilde C_2$ with unit norm such that $\|\tilde w_1 - \tilde w_2\| < \varepsilon$.
\end{lem}

\begin{proof}
	
Let us fix an arbitrary $\varepsilon\in (0, 1)$. We consider the vector $\tilde v := (\mathbf{0}, -1)\in X\times \mathbb{R}$. Now the density of $\tilde C_1 - \tilde C_2$ yields the existence of two vectors $\tilde v_i = (v_i, r_i) \in \tilde C_i$, $i=1,2$,  such that
	$$\|\tilde v - \left(\tilde v_1 - \tilde v_2\right)\| < \frac \varepsilon 2 \, ,  $$
	hence
	\begin{equation}\label{est1}
	\|v_1-v_2\| < \frac \varepsilon 2 \mbox{ and } |-1-(r_1-r_2)| = |r_1 - (r_2-1)| < \frac \varepsilon 2 \, .
	\end{equation}
	Due to the definition of $\tilde C_2$ and that $(v_2, r_2) \in \tilde C_2$, we have $(v_2, r_2-1) \in \tilde{C_2}$. Also,
	\begin{equation}\label{est2}
	\|(v_2, r_2-1)\| \ge |r_2-1| \ge 1
	\end{equation}
	since $r_2\le 0$. Moreover, $|r_1|\ge |r_2-1|-\varepsilon/2>1/2$.
	
	Let us set
	$$\tilde w_1:= \frac{(v_1, r_1)}{\|(v_1, r_1)\|} \in \tilde C_1 \mbox{ and } \tilde w_2:= \frac{(v_2, r_2-1)}{\|(v_2, r_2-1)\|} \in \tilde C_2 \, . $$
	Apparently, $\| \tilde w_1 \| =1$ and $\| \tilde w_2 \| =1$ .
	Using \eqref{est1} and \eqref{est2}, we estimate
	\begin{align*}
	\|\tilde w_1 &- \tilde w_2\|=\Big\| \frac{(v_1, r_1)}{\|(v_1, r_1)\|} - \frac{(v_2, r_2-1)}{\|(v_2, r_2-1)\|} \Big\|\\
	& \le \Big\| \frac{(v_1, r_1)}{\|(v_1, r_1)\|} - \frac{(v_1, r_1)}{\|(v_2, r_2-1)\|} \Big\| + \Big\| \frac{(v_1, r_1)}{\|(v_2, r_2-1)\|} - \frac{(v_2, r_2-1)}{\|(v_2, r_2-1)\|} \Big\| \\
	&= \|(v_1, r_1)\| \Big| \frac{1}{\|(v_1, r_1)\|} - \frac{1}{\|(v_2, r_2-1)\|} \Big| + \frac{\|(v_1, r_1)-(v_2, r_2-1)\|}{\|(v_2, r_2-1)\|}\\
	&=\frac{|\|(v_2, r_2-1)\|-\|(v_1, r_1)\||}{\|(v_2, r_2-1)\|} + \frac{\|(v_1, r_1)-(v_2, r_2-1)\|}{\|(v_2, r_2-1)\|} \\
	&\le 2 \frac{\|(v_1, r_1)-(v_2, r_2-1)\|}{\|(v_2, r_2-1)\|} \le 2{\max \{ \|v_1-v_2\|, |r_1 - (r_2-1)| \}} < \varepsilon \, .
	\end{align*}
	The proof is complete.
\end{proof}

\begin{thm}[Lagrange multiplier rule]\label{Lagrange}
	Let us consider the optimization problem
	$$f (x)\to \min \ \mbox{ subject to } \ x \in S \ ,$$
	where $f:X \longrightarrow \mathbb{R}  \cup \{+\infty\}$ is lower semicontinuous and proper and $S$ is a closed subset of the Banach space $X$.
	Let $x_0$ be a solution of the above problem.
	Let $\tilde C_{epi f} (x_0, f(x_0))$ and $C_S(x_0)$ be closed convex cones, contained in the corresponding Bouligand approximating cones $T_{epi f}(x_0, f(x_0))$ and $T_S(x_0)$. Let at least one of them consist of derivable tangent vectors. \\
	(a) If $\tilde C_{epi f}(x_0, f(x_0))- C_S(x_0) \times (-\infty, 0]$ is not dense in $X\times \mathbb{R}$, then
	there exists a   pair
	$(\xi, \eta)\in
	X^*  \times  \mathbb{R}$ such that
	\begin{enumerate}
		\item [(i)] $(\xi,\eta)\not = (\mbox{\bf 0},0)$;
		\item [(ii)] $\eta\in  \{0,1\}$;
		\item [(iii)] $\langle \xi ,v\rangle\ \le 0$  for every $v \in C_{S}(x_0)$;
		\item [(iv)] $\langle \xi, w \rangle +\eta s\ge 0$   for every $(w,s)\in \tilde C_{epi f}(x_0, f(x_0))$.
	\end{enumerate}
	(b) If $\tilde C_{epi f}(x_0, f(x_0))- C_S(x_0) \times (-\infty, 0]$ is dense in $X\times \mathbb{R}$, then $epi f$ and $S \times (-\infty, f (x_0)]$ are not  subtransversal at $(x_0, f(x_0))$.
\end{thm}

\begin{proof}
	(a) If $\tilde C_{epi f}(x_0, f(x_0))- C_S(x_0) \times (-\infty, 0]$ is not dense in $X\times \mathbb{R}$, then there exist $(\bar x, \bar r) \in X\times \mathbb{R}$ and $d>0$ such that
	$$\left(\overline{\tilde C_{epi f}(x_0, f(x_0))- C_S(x_0) \times (-\infty, 0]} \right) \cap \left((\bar x, \bar r) +d {\mbox{\bf B}}_{X\times \mathbb{R}} \right) = \emptyset \, .$$
	Then,
	$$\tilde C \cap \tilde  D = \emptyset \, ,$$
	where $\tilde C := \overline{\tilde C_{epi f}(x_0, f(x_0))- C_S(x_0) \times (-\infty, 0]} $ is a closed convex cone and $$\tilde D:= \{ (x, r) \in X\times\mathbb{R} \ | \ (x, r) = \alpha((\bar x, \bar r) + d(x_1, r_1)), \ \alpha>0, \ (x_1, r_1) \in {\mbox{\bf B}}_{X\times \mathbb{R}} \}$$ is an open convex cone (non-empty). We can separate $\tilde C$ and $\tilde D$ and find a non-zero  pair
	$(\xi, \eta)\in X^*  \times  \mathbb{R}$ and a real $\alpha$ such that
	$$\langle \xi, v_1\rangle + \eta r_1 \ge \alpha > \langle \xi, v_2\rangle + \eta r_2$$
	for all $(v_1, r_1) \in \tilde C$ and $(v_2, r_2) \in \tilde D$. Since $(\mathbf{0}, 0)$ lies in $\tilde C$ and on the boundary of $\tilde D$, we have that $\alpha=0$. Hence,
	$$\langle \xi, v_1\rangle + \eta r_1 \ge 0 $$
	for all $(v_1, r_1) \in \tilde C$, which is
	$$\langle \xi, v'-v''\rangle + \eta (r'-r'') \ge 0 $$
	for all $(v', r') \in \tilde C_{epi f}(x_0, f(x_0))$ and $(v'', r'') \in C_S(x_0) \times (-\infty, 0]$. By taking $v'=v''=\mathbf{0}$, $r'=0$ and $r''<0$ we obtain that $\eta \ge0$. Hence without loss of generality we
	may assume that $\eta\in  \{0,1\}$. By taking $v'=\mathbf{0}$, $v'' = v \in C_S(x_0)$ and $r'=r''=0$ we obtain that $\langle \xi ,v\rangle\ \le 0$. By taking $(v', r') = (w, s) \in \tilde C_{epi f}(x_0, f(x_0))$ and $(v'', r'')=(\mathbf{0},0)$, we obtain that  $\langle \xi, w \rangle +\eta s\ge 0$.
	
	(b) Let $\tilde C_{epi f}(x_0, f(x_0))- C_S(x_0) \times (-\infty, 0]$ be dense in $X\times \mathbb{R}$.
	
	Without loss of generality we may assume that $x_0$ is a strong minimum of $f$ on $S$. This is due to the fact that if $g:X \to \mathbb{R}  \cup \{+\infty\}$ is strictly Fr\'echet differentiable at $x_0$, $g(x_0)=0$ and $g'(x_0)=\mathbf{0}$, then
	$$T_{epi f}(x_0, f(x_0)) = T_{epi (f+g)}(x_0, f(x_0)) \mbox{ and } G_{epi f}(x_0, f(x_0)) = G_{epi (f+g)}(x_0, f(x_0)) \, .$$
	Indeed, $(v_0, r_0) \in T_{epi f}(x_0, f(x_0))$ if and only if there exist sequences $(v_m, r_m) \to (v_0, r_0)$ and $t_m \searrow 0$ such that
	$$(x_0, f(x_0)) + t_m(v_m, r_m) \in epi f $$
	which is equivalent to
	$$\frac{f(x_0+t_m v_m) - f(x_0)}{t_m} \le r_m \, .$$
	Let us denote
	$$r'_m := \frac{g(x_0+t_m v_m) - g(x_0)}{t_m} = \frac{g(x_0) +  \langle g'(x_0), t_m v_m \rangle +o(\|t_m v_m\|) - g(x_0)}{t_m} \to 0 \, .$$ Then,
	$$\frac{(f+g)(x_0+t_m v_m) - (f+g)(x_0)}{t_m} \le r_m + r'_m$$
	which is equivalent to
	$$(x_0, (f+g)(x_0)) + t_m(v_m, r_m+r'_m) \in epi (f+g) $$
	for the sequences  $(v_m, r_m+ r'_m) \to (v_0, r_0)$ and $t_m \searrow 0$. This verifies that $T_{epi f}(x_0, f(x_0)) \subset T_{epi (f+g)}(x_0, f(x_0))$. As $-g$ satisfies the same assumptions as $g$, the reverse inclusion is verified as well. The proof for derivable tangent cones is analogous. By putting $g(x):=\|x-x_0\|^2$, we obtain that
	$$ \tilde C_{epi f}(x_0, f(x_0)) = C_{epi (f+g)}(x_0, f(x_0)) $$
	and $x_0$ is a strong minimum of $f+g$ on $S$.
	
	Let us assume that $epi f$ and $\tilde S:= S \times (-\infty, f (x_0)]$ are subtransversal at $(x_0, f(x_0))$ with constant $K>0$. By applying Lemma \ref{nonsepcones} for $\varepsilon:=\frac{1}{K}$ and then Proposition \ref{nonsep}, we obtain that the sets $epi f$ and $\tilde  S$ can not be separated. That is, there exists a sequence $(x_m, r_m)\in epi f \cap \tilde  S$ converging to $(x_0, f (x_0))$ such that $(x_m,r_m) \neq (x_0, f (x_0))$ for every positive integer $m$.  But $(x_m, r_m)\in epi f \cap \tilde  S$ implies that $r_m\ge f(x_m)$   and     $r_m\le f(x_0)$. Because
$x_0$ is a strong local minimum of $f$ on $S$,  for each
sufficiently large $m$ the following inequalities hold true
$
r_m \ge f(x_m) > f(x_0) \ge r_m,
$
 a contradiction.
	
	Therefore $epi f$ and $\tilde S:= S \times (-\infty, f (x_0)]$ are not  subtransversal at $(x_0, f(x_0))$, which completes the proof.
\end{proof}

\section{Intersection properties}

Let $A$ and $B$ be two smooth manifolds and $x_0 \in A\cap B$.
The classical meaning of transversality in this  case is that the  tangent space   to the   manifold $A\cap B$ at the point
$x_0$   equals the intersection of the tangent spaces to $A$ and $B$, respectively, at $x_0$.
Next we   obtain   some tangential intersection properties as corollaries of subtransversality.

\begin{prop}[Intersection property with respect to Bouligand and derivable tangent cones]\label{boul}
	Let $A$ and $B$ be closed subsets of the Banach space $X$ and let
	$A$ and $B$ be subtransversal at $x_0 \in A \cap B$.
	Then, 
	$$ T_{A}(x_0) \cap  G_{B}(x_0) \subset  T_{A\cap B}(x_0)\, , $$
	where $ T_{A}(x_0)$ ($ T_{A\cap B}(x_0)$)  is the Bouligand tangent cone to $A$ ($A\cap B$) at $x_0$ and $G_{B}(x_0)$ is the derivable  tangent cone to $B$ at $x_0$. Moreover,
	$$ G_{A}(x_0) \cap  G_{B}(x_0) =  G_{A\cap B}(x_0)\, . $$
\end{prop}

\begin{proof}
	Let $v_0$ be in $T_{A}(x_0) \cap  G_{B}(x_0)$. Without loss of generality, we may assume that $\|v_0\|=1$.  Since  $v_0$ belongs to the Bouligand tangent cone to $A$ at $x_0$, we have that there exist sequences $t_m \searrow 0$ and $v_m^A \to v_0$ such that
	$x_0 +t_m v_m^A \in A \, . $
	Since $v_0$  belongs to the derivable  tangent cone to $B$ at $x_0$, we have that for all small enough $t>0$ there exists $v_t^B \in X$, such that $x_0 +t v_t^B \in B$ and $v_t^B \to v_0$ as ${t\searrow 0}$.
	Therefore, $x_0 +t_m v_m^B \in B$ for $m\in \mathbb{N}$ large enough and $v_m^B \to v_0$
	(here $ v_m^B:=v_{t_m}^B $).
	
	Let us fix an arbitrary positive $\varepsilon$. Let $K$ and $\delta$ be the constants from the definition of subtransversality.
Then, there exists $m_0\in \mathbb{N}$ such that
	$$ \|v_m^A-v_0\| \le \frac{\varepsilon}{2K+3} \mbox{ and } \|v_m^B-v_0\| \le \frac{\varepsilon}{2K+3} $$
	for all $m\ge m_0$.    Let $m_1\ge m_0$ be such that
	$$t_{m} \le \frac{\delta}{1+\varepsilon} \mbox{ for all } m\ge m_1$$
	and let us denote
	$$ x^A_m := x_0 +t_{m} v_{m}^A \in A \mbox{ and } x^B_m := x_0 +t_{m} v_{m}^B \in B \mbox{ for all } m\ge m_1 \, . $$
	It is straightforward  that
	$$\|x^A_m-x^B_m\| = t_m\|v_{m}^A - v_{m}^B\|\le t_m\left(\|v_{m}^A -v_0\| + \| v_0- v_{m}^B\|\right)\le \frac{2\varepsilon}{2K+3} t_m \, .$$
	
	Since $$\|x^A_m-x_0\|=t_m\|v_m^A\|\le t_m\left(\|v_0\|+ \frac{\varepsilon}{2K+3}\right)
	\le \frac{\delta}{1+\varepsilon} \left(1+
	\frac{\varepsilon}{2K+3}\right)<\delta$$
	and analogously
	$\|x^B_m-x_0\|<\delta,$ for  $m\ge m_1$ we have
	$$d\left( x^A_m, A\cap B\right) \le K\left( d\left( x^A_m, A\right)+d\left( x^A_m, B\right)\right) \le K\cdot d\left( x^A_m, x^B_m\right) = K \left\| x^A_m -x^B_m\right\|\, . $$
	From the definition of a distance from a point to a set there exists $x^{AB}_m \in A \cap B$ with
	$$\left\|x^{AB}_m - x^A_m\right\| \le d\left( x^A_m, A\cap B\right) + \left\| x^A_m -x^B_m\right\| \le  ( K+1) \left\| x^A_m -x^B_m\right\|\ .$$
Note that if $x^A_m=x^B_m$ we just put $x^{AB}_m$ to coincide with these points and all addends are zero.
  We  estimate
	\begin{align*}
	\|x^{AB}_m& - ( x_0 + t_m v_0)\| = \|x^{AB}_m - (x_0 + t_m v_m^A) -t_m(v_0- v_m^A)\| \\
	&\leq \|x^{AB}_m -x^A_m\| + t_m\|v_0- v_m^A\| \le ( K+1) \left\| x^A_m -x^B_m\right\| +t_m \frac{\varepsilon}{2K+3}  \\
	&\leq ( K+1)\cdot \frac{2\varepsilon}{2K+3}\, t_m +  \frac{\varepsilon}{2K+3}\, t_m = \varepsilon t_m \, .
	\end{align*}
	Hence, for $m\ge m_1$, the following is true
	$$ x^{AB}_m \in x_0 + t_m(v_0+ \varepsilon \bar {\mbox{\bf B}}) \, .$$
	We have obtained that for every $v_0 \in  T_A(x_0) \cap G_B(x_0)$, $\|v_0\|=1$ and for every $\varepsilon >0$ there exists
	$m_1 \in \mathbb{N}$ such that
	$$(A\cap B) \cap \left( x_0 + t_m(v_0+ \varepsilon \bar {\mbox{\bf B}})\right) \neq \emptyset $$
	for all $m\ge m_1$. From this, it follows that $ T_A(x_0) \cap G_B(x_0)$ is a subset of $T_{A\cap B}(x_0)$.
	
	The inclusion $ G_{A}(x_0) \cap  G_{B}(x_0) \subset  G_{A\cap B}(x_0) $ can be proved in the same way. Then, the
equality for the derivable cones follows from their monotonicity.
\end{proof}

\begin{prop}[Intersection property with respect to Clarke tangent cones]\label{clarke_t}
	Let $A$ and $B$ be closed subsets of the Banach space $X$ and let
	$A$ and $B$ be subtransversal at $x_0 \in A \cap B$.
	Then, 
	$$\hat T_{A}(x_0) \cap \hat T_{B}(x_0) \subset \hat T_{A\cap B}(x_0)\, , $$
	where $\hat T_S(x_0)$ is the Clarke tangent cone to $S$ at $x_0$.
\end{prop}

\begin{proof}
	Let us fix a positive $\varepsilon$ and an arbitrary $v_0 \in \hat T_{A}(x_0) \cap \hat T_{B}(x_0)$. Let $\delta >0$  and $K >0$ be the constants from the definition of subtransversality of $A$ and $B$ at $x_0$.
	
	From the definition of Clarke tangent cone for $v_0$, we have that for $\eta := \frac{\varepsilon }{2K+3}>0$ there exists $\delta_A >0 $  such that
	for all $x \in (x_0 + \delta_A \bar {\mbox{\bf B}})\cap A$ and for all $t \in (0, \delta_A)$ it holds true that
	$$(x+t(v_0+\eta \bar {\mbox{\bf B}})) \cap A \neq \emptyset$$ and correspondingly, there exists $\delta_B>0$ such that for all $x \in (x_0 + \delta_B \bar {\mbox{\bf B}})\cap B$ and for all $t \in (0, \delta_B)$ it holds true that
	$$(x+t(v_0+\eta  \bar {\mbox{\bf B}})) \cap B \neq \emptyset \, .$$
	
	Let us fix $\bar \delta := \min\{\frac \delta 2, \delta_A, \delta_B\}>0$ and an arbitrary $\bar x \in (x_0 + \bar  \delta \bar {\mbox{\bf B}})\cap (A \cap B)$.  Let $h_0$ be an arbitrary positive real satisfying
	\begin{equation}\label{h0}
	h_0 \le \bar h:=  \min\left\{\bar \delta,\,  \frac{\delta}{2(\eta +\|v_0\|)}\right\} \, .
	\end{equation}
	We obtain that there exist vectors $v_0^A \in X$ and $v_0^B \in X$ such that
	\begin{equation}\label{vTilde}
	\|v_0^A -  v_0 \| \le  \eta , \ \|v_0^B -  v_0\| \le  \eta  \,
	\end{equation}
	and
	$$ x^A := \bar x + h_0 v_0^A \in A, \ x^B := \bar x + h_0 v_0^B \in B \, . $$
	
	Taking into account \eqref{vTilde}, we can verify directly
	that
	\begin{align*}
	\|x^A - x^B\| = h_0 \|v_0^A - v_0^B\|= h_0\|v_0^A - v_0 + v_0 -v_0^B \| \le 2 \eta  h_0
	\end{align*}
	and
	$$\|x^A - \bar x\| =  h_0 \|v_0^A \| = h_0\|v_0^A -
	v_0 +  v_0\|  \le h_0(\eta  + \|  v_0\|) \, .$$
	Analogously, we obtain that
	$$ \|x^B - \bar x\| \le h_0(\eta  + \|  v_0\|) \, .$$
	
	We have that
	$$
	\|x^A  -  x_0\|  \le \|x^A - \bar x\| + \|\bar x - x_0\|
	\le h_0(\eta  + \|  v_0\|) + \bar \delta  \le \frac \delta 2 +  \bar \delta
	\le \delta
	$$
	using the estimate of $h_0$ and the definition of $ \bar \delta$. Analogously,
	$\|x^B -  x_0\|\le \delta$.
	
	Therefore,
$$d\left( x^A, A\cap B\right) \le K\left( d\left( x^A, A\right)+d\left( x^A, B\right)\right) \le K\cdot d\left( x^A, x^B\right) = K \left\| x^A  -x^B \right\|\, . $$
	From the definition of a distance from a point to a set there exists $x^{AB} \in A \cap B$ with
	$$\left\|x^{AB} - x^A\right\| \le d\left( x^A, A\cap B\right) + \left\| x^A -x^B\right\| \le  ( K+1) \left\| x^A -x^B\right\|\ .$$
Note that if $x^A=x^B$ we just put $x^{AB}$ to coincide with these points and all addends are zero.

 We  estimate
	\begin{align*}
	\|x^{AB}& - (\bar x + h_0 v_0)\| = \|x^{AB} - (\bar x + h_0 v_0^A) -h_0(v_0- v_0^A)\| \\
	&\leq \|x^{AB} -x^A\| + h_0\|v_0- v_0^A\| \le ( K+1) \left\| x^A -x^B\right\|\ +h_0 \eta  \\
	&\leq ( K+1) 2\eta h_0 +h_0 \eta  = h_0\eta \left( 2K+3\right) = \varepsilon h_0 \, .
	\end{align*}
	Hence, 
	$$ x^{AB} \in \left(A\cap B\right)\cap \left( \bar x + h_0(v_0+ \varepsilon \bar {\mbox{\bf B}}) \right) \, .$$
	
	We have obtained that for every $v_0 \in \hat T_A(x_0) \cap \hat T_B(x_0)$ and for every $\varepsilon >0$ there exists
	$\bar \delta>0$ such that for each point $\bar x \in (A\cap B) \cap (x_0 + \bar \delta \bar {\mbox{\bf
			B}})$, there exists $\bar h$ such that
	$$(A\cap B) \cap (\bar x+ h_0 (v_0 + \bar \varepsilon \bar {\mbox{\bf B}})) \neq \emptyset$$
	for each $ h_0 \in [0, \bar h]$. Therefore, $v_0 \in \hat T_{A\cap B}(x_0)$. This completes the proof.
\end{proof}

\begin{rem}
It is remarkable   that the same intersection properties (cf. Proposition \ref{boul} and Proposition \ref{clarke_t})
  appear in \cite{AF91} 	back in 1990.    Proposition \ref{Lelincho} shows that
the ``local 	transversality condition'' assumed  in    \cite{AF91} is a sufficient
	condition for tangential transversality. Hence, because tangential transversality implies subtransversality,
	Corollary 4.3.5 in \cite{AF91} follows from our Proposition \ref{boul} and Proposition \ref{clarke_t}.
\end{rem}

\begin{prop}\label{Lelincho} Let $A$ and $B$ be closed subsets of the Banach space $X$ and  let $x_0\in A\cap B$.
  If there exist constants $\delta >0$, $\alpha\in [0,1)$   and  $ M>0$   such that for each  $   x^A \in (x_0 + \delta \bar {\mbox{\bf B}})\cap A$
    and for each $ x^B \in (x_0 + \delta \bar {\mbox{\bf B}})\cap B $  it is true that $ \bar {\mbox{\bf B}} \subset \left( G_A(x^A)\cap M\bar {\mbox{\bf B}}\right) -T_B(x^B)+\alpha \bar {\mbox{\bf B}}$, then the sets $A$ and $B$ are tangentially transversal at $x_0$.
\end{prop}

\begin{proof}
Let us fix an arbitrary positive real $\eta <1-\alpha$ and check that $A$ and $B$ are tangentially transversal at $x_0$ with constants $\delta$, $M+3$ and $\eta$.

Let us choose arbitrary $   x^A \in (x_0 + \delta \bar {\mbox{\bf B}})\cap A$ and $ x^B \in (x_0 + \delta \bar {\mbox{\bf B}})\cap B $
with $x^A\not= x^B$. Then the vector
$$v:= \frac{x^B-x^A}{\| x^B-x^A\|}$$
is of norm one, and therefore there exist vectors $w^A\in G_A(x^A)$, $\| w^A\| \le M$ and $w^B\in T_B(x^B)$ such that
$\left\| v-\left( w^A-w^B\right)\right\| \le \alpha$.
Then $w^B\in T_B(x^B)$ implies the  existence of sequences $t_m \searrow 0$ and $w_m^B \to w^B$ such that
	$x^B +t_m w_m^B \in B \, . $
	Since $w^A$  belongs to the derivable  tangent cone to $A$ at $x^A$, we have that for all small enough $t>0$ there exists $w_t^A \in X$, such that $x^A +t w_t^A \in A$ and $w_t^A \to w^A$ as ${t\searrow 0}$.
	Therefore, $x^A +t_m w_m^A \in A$ for $m\in \mathbb{N}$ large enough and $w_m^A \to w^A$
	(here $ w_m^A:=w_{t_m}^A $). Moreover,
$$\| x^A-x^B +t_m(w_m^A-w_m^B)\| \le $$ $$\le \| x^A-x^B + t_m v\| +t_m\|w^A-w^B - v\|+t_m\|w_m^A-w^A\|+t_m\|w_m^B-w^B)\| \le$$
$$\le \| x^A-x^B\| -t_m +t_m \alpha +t_m\|w_m^A-w^A\|+t_m\|w_m^B-w^B \| \ .$$
Then for all $m$ big enough (for which $\|w_m^A-w^A\|\le (1-\alpha-\eta)/2$, $\|w_m^B-w^B\|\le (1-\alpha-\eta)/2$) the estimate
 $$\| x^A-x^B +t_m(w_m^A-w_m^B)\| \le \| x^A-x^B\| -t_m \eta$$
 holds true. It remains to note that
 $$\| w_m^A\| \le \| w^A\| +\frac{1-\alpha -\eta}{2}\le M+1 \mbox{ and } $$ $$\| w_m^B\| \le \| w^B\| +\frac{1-\alpha -\eta}{2}\le \alpha +\| w^A\| +\| v \| +1\le M+3 \ .$$
\end{proof}

\begin{rem}
	It is also remarkable that in 1982 subtransversality is proven to be a sufficient condition for a tangential intersection property for Dubovitzki-Milyutin tangent cones (even with equality) by Dolecki in \cite{Dolecki}. The word ``subtransversality'' is not mentioned, but the distance inequality from its definition is used instead.
\end{rem}

\section{Massive sets}

The classical concept of compactly epi-Lipschitz sets in Banach spaces  was introduced by J.M. Borwein and H.M. Strojwas in 1985 in \cite{BS85} as appropriate for investigating tangential approximations of the Clarke tangent cone in Banach spaces. Since then, it has been an important notion in nonsmooth analysis and has been frequently used in qualification conditions for obtaining normal intersection properties and calculus rules concerning limiting normal Fr$\acute{\mbox{\rm e}}$chet cones and subdifferentials (in Asplund spaces, cf. \cite{Mord} and \cite{Penot}) and $G$-normal cones and $G$-subdifferentials (in general Banach spaces, cf. \cite{Ioffe}). Compactly epi-Lipschitz sets are called \textit{massive} in \cite{IoffeBook}. Here is the corresponding

\begin{definition}
	Let $A$ be a closed subset of the Banach space $X$ and $x_0\in A$. We say that $A$ is compactly epi-Lipschitz (massive) at $x_0$, if there exist $\varepsilon>0$, $\delta >0$ and a compact set $K \subset X$, such that for all $x \in A \cap (x_0 + \delta \bar {\mbox{\bf B}})$, for all $v \in X$, $\|v\|\le \varepsilon$ and for all $t \in [0, \delta]$, there exists $k \in K$, for which $x + t(v-k)\in A$.
\end{definition}

Using the concept of   massive sets, we are able to prove the following sufficient condition for tangential transversality.

\begin{thm}\label{almostm}
	Let $A$ and $B$ be closed subsets of the Banach space $X$ and let $x_0 \in A \cap B$.  Let $A$ be   massive and $\hat T_{A}(x_0)- \hat T_B(x_0)$ be dense in $X$. Then $A$ and $B$ are tangentially transversal at $x_0$.
\end{thm}

\begin{proof}Let $\varepsilon>0$, $\delta >0$ and  the set $K$  be those from the definition of $A$ --  massive at $x_0$.
Let $q\in (0,1)$ be arbitrary. Because the set $K$ is compact and $\varepsilon>0$, there exists a finite
  $\varepsilon q - $net $F:= \{k_1, k_2, \dots, k_n \}$ for $K$.
	    Let us set $\eta := \frac{\varepsilon(1-q)}{4}$.
	
	Due to the density of $\hat T_A(x_0)-\hat T_B(x_0)$ in $X$, we obtain that for all $s \in \{1, \dots, n\}$, there exist $w_s^A \in \hat T_A(x_0)$ and $w_s^B \in \hat T_B(x_0)$ such that
	\begin{equation*}\label{dens}
	\|k_s - (w_s^A - w_s^B)\| \leq \eta \, .
	\end{equation*}
	
	From the definition of Clarke tangent cone for $w_s^A \in \hat T_A(x_0)$, we have that there exists $\delta_s^A>0$ such that for all $x \in (x_0 + \delta_s^A \bar {\mbox{\bf B}})\cap A$ and for all $t \in (0, \delta_s^A)$ it holds true that
	$$(x+t(w_s^A+\eta \bar {\mbox{\bf B}})) \cap A \neq \emptyset \, .$$
	Analogously, for $w_s^B \in \hat T_B(x_0)$, we have that there exists $\delta_s^B>0$ such that for all $x \in (x_0 + \delta_s^B \bar {\mbox{\bf B}})\cap B$ and for all $t \in (0, \delta_s^B)$ it holds true that
	$$(x+t(w_s^B+\eta \bar {\mbox{\bf B}})) \cap B \neq \emptyset \, .$$
	
	We set $N:= \max\{\|k_s\| \, : \,  s = 1, \dots, n \}$, $M:= \max\{\|w_s^A\|, \|w_s^B\| \, : \,  s = 1, \dots, n \} + N + \varepsilon(1+q)+ \eta$
	and
	$$\bar \delta:=\min\{\varepsilon, \delta, \delta_s^B, \frac{\delta_s^A}{1+N +2\varepsilon} \, : \,  s = 1, \dots, n\} \, .$$
	Let  $x^A \in (x_0 + \bar \delta \bar {\mbox{\bf B}})\cap A$ and  $x^B \in (x_0 + \bar \delta \bar {\mbox{\bf B}})\cap B$ with $x^A\not=x^B$ and let  $t\in (0, \min \{\bar \delta, \frac{\|x^A-x^B\|}{\varepsilon} \})$ be arbitrary.
	Let us set
	$$v:=-\frac{x^A-x^B}{\|x^A-x^B\|} \, .$$
	Then, $\|\varepsilon v\| = \varepsilon$, $0<t< \bar \delta \le \varepsilon$, $x^A \in (x_0 + \bar \delta \bar {\mbox{\bf B}})\cap A \subset (x_0 + \delta \bar {\mbox{\bf B}})\cap A$ and therefore there exists $k\in K$ such that
	$$\tilde x^A := x^A +t(\varepsilon v -k) \in A \, .$$
	Since $k\in K$, then $\|k-k_s\| \le \varepsilon q$ for some $s \in \{ 1, \dots, n \}$. We estimate
	\begin{align*}
	\|\tilde x^A - x_0\| &\le \|x^A-x_0\| + t\|\varepsilon v - k\|  \le \bar \delta + \bar \delta (\|\varepsilon v\| + \|k_s\| +\| k -k_s \|)\\
	&\le \bar \delta (1+ \varepsilon+N+\varepsilon q) <  \bar \delta (1+N+2\varepsilon) \le \delta_s^A
	\end{align*}
	and therefore
	$$(\tilde x^A  + t (w_s^A+ \eta \bar {\mbox{\bf B}})) \cap A \ne \emptyset \, . $$
	Then, there exists $w^A \in X$, $\|w^A-w_s^A\| \le \eta$, such that $ \tilde x^A + t w^A \in A$ and we obtain
	\begin{align*}
	\tilde x^A + t w^A = x^A + t(\varepsilon v -k) + t w^A = x^A + t (\varepsilon v -k +w^A) \in A
	\end{align*}
	and
	$$\|\varepsilon v -k +w^A\| = \|\varepsilon v -k_s+(k_s -k) +(w^A-w^A_s) + w^A_s \| \leq \varepsilon +  N +\varepsilon q+ \eta + \|w^A_s\| \leq M \, .$$
	
	For $x^B \in (x_0 + \bar \delta \bar {\mbox{\bf B}})\cap B \subset (x_0 + \delta_s^B \bar {\mbox{\bf B}})\cap B$ we have that
	$$ (  x^B  + t (w_s^B+ \eta \bar {\mbox{\bf B}})) \cap B \ne \emptyset \, ,$$
	which implies that there exists $w^B \in X$, $\|w^B-w_s^B\| \le \eta$, such that $$ x^B + t w^B \in B \, .$$
	Obviously $\|w^B\|\leq \|w^B_s\|+ \eta < M$.
	
	We estimate
	\begin{align*}
	\|(x^A& +  t \left(\varepsilon v - k +  w^A \right))-(x^B +  t  w^B)\| \\
	&=\|x^A -x^B+  t \varepsilon v  +  t(w^A - k  -  w^B)\| \le
	\|x^A -x^B+  t \varepsilon v\|  +  \|t(w^A - k  -  w^B)\|\\
	&\le \Big\|x^A -x^B-t\varepsilon \frac{x^A-x^B}{\|x^A-x^B\|}  \Big\| + \\
	& \ \ \ \ \ \ \ \ \ \ + t\|w^A-w^A_s
	+  w^B_s-w^B +k_s-k + (w^A_s-w^B_s - k_s) \|\\
	&\le \|x^A -x^B\| \Big|1-\frac{t\varepsilon}{\|x^A-x^B\|}  \Big| +\\
	& \ \ \ \ \ \ \ \ \ \ + t\left(\|w^A-w^A_s\| +
	\| w^B_s-w^B\| +\|k_s-k\| + \|(w^A_s-w^B_s) - k_s \|\right)\\
	&\le \|x^A -x^B\|-t\varepsilon+t(3\eta+\varepsilon q)=\|x^A -x^B\|-t\left(\varepsilon-\varepsilon q -3\frac{\varepsilon(1-q)}{4}\right)\\
	&= \|x^A -x^B\|-t\eta \, ,
	\end{align*}
	where $\eta:= \frac{\varepsilon(1-q)}{4}>0$.
	
	This verifies the definition of $A$ and $B$ -- tangentially transversal at $x_0$ with constants $M>0$, $\delta >0$ and  $\eta >0$.
\end{proof}

\begin{corol}\label{almostn}
	Let $A$ and $B$ be closed subsets of the Banach space $X$ and let $x_0 \in A \cap B$.  Let $A$ be   massive and $\hat T_{A}(x_0)- \hat T_B(x_0)$ be dense in $X$. Then,
	\begin{equation}\label{Gnor}
	N_{A\cap B}(x_0) \subset N_A(x_0) + N_B(x_0) \, ,
	\end{equation}
	where $N_S(x)$ is the $G$-normal cone to the set $S$ at the point $x$.
\end{corol}

\begin{proof}
	Due to Theorem \ref{almostm}, $A$ and $B$ are tangentially transversal. Tangential transversality implies subtransversality  due to Proposition \ref{subtr} and subtransversality implies \eqref{Gnor} due to Theorem 7.13 in \cite{IoffeBook}.
\end{proof}

Let $f_1: X  \rightarrow \mathbb{R}\cup \{+\infty\}$ and  $f_2: X  \rightarrow \mathbb{R}\cup \{+\infty\}$ be lower semicontinuous and proper and  $x_0 \in X$ be in $dom f_1 \cap dom f_2$. We are going to apply the results from this section to the closed sets
$$C_1 := \{ (x, r_1, r_2) \in X\times \mathbb{R}\times \mathbb{R} \ | \ r_1 \ge f_1(x) \}$$
and
$$C_2 := \{ (x, r_1, r_2) \in X\times \mathbb{R}\times \mathbb{R} \ | \ r_2 \ge f_2(x) \} $$
in order to obtain a sum rule for the $G$-subdifferential. This is the approach introduced by Ioffe in \cite{Ioffe84}. We will need the following technical lemma.

\begin{lem}\label{qual_cons}
	The following are equivalent
	\begin{enumerate}[(i)]
		\item ${\hat T_{epi f_1}(x_0, f_1(x_0))- \hat T_{epi f_2}(x_0,  f_2(x_0))} $ is dense in $ X\times \mathbb{R}$
		\item ${\hat T_{C_1}(x_0, f_1(x_0), f_2(x_0))- \hat T_{C_2}(x_0, f_1(x_0), f_2(x_0))} $ is dense in $ X\times \mathbb{R}\times \mathbb{R}$
		\item $\{N_{C_1}^C(x_0, f_1(x_0), f_2(x_0)) \} \cap \{-N_{C_2}^C(x_0, f_1(x_0), f_2(x_0)) \} = \{ (\mathbf{0}, 0, 0) \} \, ,$ where $N_S^C(x)$ is the Clarke normal cone to the set $S$ at the point $x$
		\item $\{\partial_C^\infty f_1(x_0) \} \cap \{-\partial_C^\infty f_2(x_0) \} = \{ \mathbf{0} \} \, ,$   where  $\partial_C^\infty$ is the Clarke singular subdifferential.
	\end{enumerate}
\end{lem}

\begin{proof}
	We have that
	$$N_{C_1}^C(x_0, f_1(x_0), f_2(x_0)) = \{ (x^*, s_1, 0) \in X^* \times \mathbb{R}\times \mathbb{R} \ | \  (x^*, s_1) \in N_{epi f_1}^C(x_0, f_1(x_0)) \} $$
	and
	$$N_{C_2}^C(x_0, f_1(x_0), f_2(x_0)) = \{ (x^*, 0, s_2) \in X^* \times \mathbb{R}\times \mathbb{R} \ | \  (x^*, s_2) \in N_{epi f_2}^C(x_0, f_2(x_0)) \} \, . $$
	The reals $s_1$ and $s_2$ in the expressions above are non-positive by polarity, since the vector $(\mathbf{0}, 1)$ is always contained in a tangent cone to the epigraph of a function. Using this and the definition of singular subdifferential, we obtain that $(iv)$ is equivalent to $(iii)$, which is equivalent to $(ii)$ by polarity.
	
	Using again that $(iv)$ holds if and only if
	$$\{N_{epi f_1}^C(x_0, f_1(x_0)) \} \cap \{-N_{epi f_2}^C(x_0, f_2(x_0)) \} = \{ (\mathbf{0}, 0) \} \, , $$
	we obtain that it is equivalent to $(i)$ by polarity.
\end{proof}

\begin{corol}\label{cor46}
	Let $f_1: X  \rightarrow \mathbb{R}\cup \{+\infty\}$ and  $f_2: X  \rightarrow \mathbb{R}\cup \{+\infty\}$ be lower semicontinuous and proper and  $x_0 \in X$ be in $dom f_1 \cap dom f_2$. Let $epi f_1$ be   massive and
	\begin{equation}\label{qual_con}
	\{\partial_C^\infty f_1(x_0) \} \cap \{-\partial_C^\infty f_2(x_0) \} = \{ \mathbf{0} \} \, .
	\end{equation}
	
	Then,
	$$\partial_G (f_1+f_2)(x_0) \subset \partial_G f_1(x_0) + \partial_G f_2(x_0) \, ,$$
	where $\partial_G$ is the $G$-subdifferential.
\end{corol}

\begin{proof}
	Let us set
	$$C_i := \{ (x, r_1, r_2) \in X\times \mathbb{R}\times \mathbb{R} \ | \ r_i \ge f_i(x) \} \mbox{ for } i=1,2 \, .$$
	
	We have that the qualification condition \eqref{qual_con} is equivalent to
	\begin{equation}\label{tang}
	\overline{\hat T_{C_1}(x_0, f_1(x_0), f_2(x_0))- \hat T_{C_2}(x_0, f_1(x_0), f_2(x_0))} = X\times \mathbb{R}\times \mathbb{R} \, .
	\end{equation}
	due to Lemma \ref{qual_cons}.
	
	Since $C_1$ is almost massive, we can apply Corollary \ref{almostn} and obtain that
	$$ N_{C_1\cap C_2}(x_0, f_1(x_0), f_2(x_0)) \subset N_{C_1}(x_0, f_1(x_0), f_2(x_0)) + N_{C_2}(x_0, f_1(x_0), f_2(x_0)) \, .$$
	It is direct that
	$$C_1 \cap C_2 \subset C:= \{ (x, r_1, r_2) \in X\times \mathbb{R}\times \mathbb{R} \ | \ r_1+r_2 \ge f_1(x) +  f_2(x) \} $$
	and by Lemma 5.5 in \cite{Ioffe89}
	\begin{align*}
	N_C(x_0, f_1(x_0), &f_2(x_0)) \subset N_{C_1\cap C_2}(x_0, f_1(x_0), f_2(x_0)) \\
	&\subset N_{C_1}(x_0, f_1(x_0), f_2(x_0)) + N_{C_2}(x_0, f_1(x_0), f_2(x_0)) \, .
	\end{align*}
	Since $x^* \in \partial_G (f_1+f_2)(x_0) \iff (x^*, -1, -1) \in N_C(x_0, f_1(x_0), f_2(x_0))  $, the proof is complete.
\end{proof}

The following statement is an abstract Lagrange multiplier rule.

\begin{corol}\label{cor47}
	Let us consider the optimization problem
	$$f (x)\to \min \ \mbox{ subject to } \ x \in S \ ,$$
	where $f:X \longrightarrow \mathbb{R}  \cup \{+\infty\}$ is lower semicontinuous and proper and $S$ is a closed subset of the Banach space $X$.
	Let $x_0$ be a solution of the above problem. If $epi f$ is   massive at $(x_0, f(x_0))$,
	then there exists a   pair
	$(\xi, \eta)\in
	X^*  \times  \mathbb{R}$ such that
	\begin{enumerate}
		\item [(i)] $(\xi,\eta)\not = (\mbox{\bf 0},0)$;
		\item [(ii)] $\eta\in  \{0,1\}$;
		\item [(iii)] $\langle \xi ,v\rangle\ \le 0$  for every $v \in \hat T_S(x_0)  $;
		\item [(iv)] $\langle \xi, w \rangle +\eta s\ge 0$   for every $(w,s)\in \hat T_{epi f}(x_0, f(x_0))$.
	\end{enumerate}
	
\end{corol}

\begin{proof}
	The corollary follows from Theorem \ref{Lagrange} and Theorem \ref{almostm}.
\end{proof}

    Corollaries \ref{almostn} and  \ref{cor46} are
known results which are obtained in a different way (cf. \cite{Ioffe89}, \cite{JouThi95}).
To the best of our knowledge, Theorem \ref{almostm} (even replacing tangential transversality by subtransversality in the
conclusion) and Corollary \ref{cor47} are new.

\section{Conclusion}

The transversality-oriented
language is extremely natural and convenient in some parts of variational analysis, including subdifferential
calculus and nonsmooth optimization. It is our understanding that the notion of subtransversality is central in many considerations in the
field. For that reason it is important to verify the subtransversality assumption in different
nontrivial cases. We view the notion of tangential transversality introduced in this paper mainly as a very useful sufficient condition for subtransversality. Nevertheless, may be some of the following open questions deserve some attention:

\begin{enumerate}
\vspace*{-0.2cm}
\item[1.] {
Tangential transversality is an intermediate property between
transversality and subtransversality (cf. Section 2).   However, the exact relation between
this new concept and the established notions of  transversality, intrinsic trans\-ver\-sa\-lity (cf. [7])
and subtransversality is not clarified yet.}

\vspace*{-0.2cm}
\item[2.] {It would be useful    to  find some  dual characterization of
tangential trans\-ver\-sa\-lity.}

\vspace*{-0.2cm}
\item[3.] It would  be  interesting to avoid the transfinite induction in the proof of Theorem
2.3 and to find     a more traditional proof.

\vspace*{-0.2cm}
\item[4.] Is there  an example of two sets satisfying the assumptions of Theorem 5.3
which are not transversal?
\end{enumerate}

\begin{center}
\end{center}



\begin{thebibliography}{00}




\bibitem{AF91} J.-P. Aubin, H. Frankowska, Set-Valued Analysis, Birkhäuser (1990)

\bibitem{BKRstt} M. Bivas, M. Krastanov, N. Ribarska, On strong  tangential transversality, preprint, 2018, https://arxiv.org/abs/1810.01814

\bibitem{BRV17} M. Bivas, N. Ribarska, M. Valkov, Properties of uniform tangent sets and Lagrange multiplier rule, Comptes rendus de l’Acade'mie bulgare des Sciences (2018), Vol 71, No7, pp.875-884

\bibitem{BS85} J.M. Borwein, H.M. Strojwas, Tangential approximations, Nonlinear Analysis: Theory, Methods \& Applications (1985)
Volume 9, Issue 12, Pages 1347-1366

\bibitem{four} F.H. Clarke, Y.S. Ledyaev, R.J. Stern, P.R. Wolenski, Nonsmooth Analysis and Control Theory,
Graduate Texts in Mathematics, Springer, New York (1998)

\bibitem{Dolecki} S. Dolecki, Tangency and differentiation, some applications of convergence theory, Ann. Math. Pura Appl, 130, 223–255 (1982)

\bibitem{DIL2015} D. Drusvyatskiy, A.D. Ioffe,  A.S. Lewis, Transversality and alternating projections for nonconvex
sets, Found. Comput. Math. 15, 1637–1651 (2015)


\bibitem{FA} M. Fabian, P. Habala, P. Hajek, V. Montesinos Santalucia, J. Pelant and V. Zizler, Functional Analysis and Infinite-Dimensional Geometry, Springer-Verlag New York (2001)


\bibitem{DG1}   V. Guillemin, A. Pollack, Differential Topology. Prentice-Hall Inc, Englewood Cliffs, N.J. (1974)

\bibitem{DG2} M. Hirsch, Differential Topology. Springer, New York (1976)

\bibitem{Ioffe84} A. Ioffe,  Approximate subdifferentials and applications I: The finite dimensional theory,
Trans. Amer. Math. Soc. 281 (1984), 389-416

\bibitem{Ioffe89} A. Ioffe, Approximate subdifferentials and applications III: The metric theory, Mathematica (1989), Volume 36, Issue 1, pp. 1-38

\bibitem{Ioffe} A. Ioffe, Transversality in Variational Analysis, J Optim Theory Appl (2017), 174(2), 343-366

\bibitem{IoffeBook} A. Ioffe, Variational Analysis of Regular Mappings: Theory and Applications, Springer Monographs in Mathematics, Springer (2017)

\bibitem{Jech} T. Jech, Set theory, Academic press (1978)

\bibitem{JouThi95} A Jourani, L Thibault, Extensions of subdifferential calculus rules in Banach spaces, Canadian Journal of Mathematics (1996), 48 (4), 834-848	

\bibitem{KR17} M. I. Krastanov and N. K. Ribarska, Nonseparation of Sets and Optimality Conditions, SIAM J. Control Optim. (2017), 55(3), 1598-1618

\bibitem{KRTs} M. I. Krastanov, N. K. Ribarska, Ts. Y. Tsachev, A Pontryagin maximum principle for
infinite-dimensional problems, SIAM Journal on Control and
Optimization, 49 (2011), No 5, 2155--2182

\bibitem{Kru15} A. Y. Kruger, N. H. Thao, Quantitative Characterizations of Regularity Properties of Collections of Sets, Journal of Optimization Theory and Applications (2015), Volume 164, Issue 1, pp 41–67

\bibitem{Kruger2018} A.Y. Kruger,  D.R. Luke,  N.H. Tao, Set regularities and feasibility
problems,  Mathematical Programming B (2018), 168, 279--311

\bibitem{LY} X. J. Li, J.  Yong, Optimal control theory for infinite
dimensional systems,  Basel, Birkh\"auser, 1994

\bibitem{Mord} B.S. Mordukhovich, Variational Analysis and Generalized Differentiation, I: Basic Theory,
Grundlehren der mathematischen Wissenschaften, Springer, New York (2006)

\bibitem{Penot} J.P. Penot, Calculus Without Derivatives, Graduate Texts in Mathematics, Springer, New York (2013)

\bibitem{R19} N. K. Ribarska, On a property of compactly epi-Lipschitz sets, Comptes rendus de l’Acad\'{e}mie bulgare des Sciences,  72 (2019), 170--173.

\bibitem{Rock80} R. T. Rockafellar, Generalized directional derivatives and subgradients of nonconvex functions, Canad. J. Math. 32(1980), 257-280

\bibitem{HS} H. Sussmann, On the validity of the transversality condition for different
concepts of tangent cone to a set,  Proceedings of the 45-th IEEE
CDC, San Diego, CA, December 13-15, 2006,  241--246.


\end{thebibliography}
\end{document}